\documentclass[12pt]{amsart}
\usepackage{graphicx}
\usepackage{tikz}        
\usepackage{amssymb}
\usepackage{epstopdf}
\usepackage[mathscr]{eucal}
\usepackage{amsmath,amssymb,amscd}
\usepackage{color}
\usepackage{epsfig}
\usepackage{bbold}
\newtheorem{theorem}{Theorem}

\newtheorem{definition}{Definition}

\newtheorem{proposition}{Proposition}
\theoremstyle{definition}
\newtheorem{remark}{Remark}

\def \mb{\mathbb}

\def \R{\mb R}                 

\def \D{\Delta}         
\def \th{\theta}       

\newcommand {\csn} {\text{csn}}

\newcommand {\sn} {\text{sn}}
\newcommand {\ctn} {\text{ctn}}

\def \S{\mb S}        
\def \H{\mb H}        
\def \M{\mb M}        

\newcommand {\q} {\mathbf{q}}

\newcommand {\F} {\mathbf{F}}

\def \and{\mbox{and}}
\usepackage[top=4cm, bottom=4cm, left=3.5cm, right=3.5cm]{geometry}
\title{ Three-dimensional  central configurations\\ in $\H^3$ and $\S^3$}
\begin{document}
\maketitle
\markboth{Suo Zhao and Shuqiang Zhu}
\author{\begin{center}
{ \bf Suo Zhao}$^{1}$,  {\bf Shuqiang Zhu$^2$}\\
\bigskip
{\footnotesize $^1$Department of Mathematics, Sichuan University, Chengdu 610064, P. R. China\\
and\\
$^2$Department of Mathematics and Statistics,
University of Victoria,
P.O.~Box 1700 STN CSC,
Victoria, BC, Canada, V8W 2Y2\\
\bigskip
saaki@163.com,  zhus@uvic.ca\\
}
\end{center}


 \begin{abstract}
  We show that  each central configuration  in the three-dimensional hyperbolic sphere is equivalent to one central configuration on a particular two-dimensional hyperbolic sphere. However, there exist both special and ordinary central configurations in the three-dimensional sphere that are not confined to any two-dimensional sphere.
 \end{abstract}

        \vspace{2mm}

        \vspace{8mm}

\section{introduction}
The curved $N$-body problem is a natural extension of the Newtonian $N$-body problem in $\R^3$ to isotropic, complete, simply connected
spaces of constant nonzero curvature, $\S^3$ and $\H^3$. For its history, we refer the readers to \cite{Dia12a}, where the equations of motion are written in extrinsic coordinates in $\R^4$ for $\S^3$, and the Minkowski space $\R^{3,1}$ for $\H^3$. This approach, different form more traditional ones like \cite{Kil98}, led to fruitful results, especially in the study of relative equilibria, which are rigid motions that become fixed points in some rotating coordinates system, \cite{DPS05,Dia12a,Dia12b,Dia13,Dia-A,DSZ-A}.

Based on the work of Diacu, especially \cite{Dia12a,Dia13}, the authors of \cite{DSZ-B} proposed to study central configurations.
Roughly speaking, central configurations are special arrangements of the point particles and the exact definition will be given later. The central configurations of the Newtonian $N$-body problem, first formulated by Laplace \cite{Lap}, are quite important in the study of the Newtonian $N$-body problem. In \cite{DSZ-B},  the authors have also showed the importance of central configurations for the curved $N$-body problem. For instance, each central configuration gives rise to a one-parameter family of relative equilibria, and central configurations are the bifurcation points in the topological classification of the curved $N$-body problem.

Some questions about these configurations were also raised in \cite{DSZ-B}.  For example, find all central configurations for $N$ point particles when $N$ is small (the three-particles case has been recently solved and will appear in a forthcoming paper). Another interesting problem is to prove (or disprove) that for generic $N$ point particles, the number of equivalent classes of central configurations is finite. Though these questions are similar to those of the Newtonian $N$-body problem  \cite{Moe-A, Sma98}, the answers are quite different in general. For example, Moulton's theorem concerning the collinear central configurations has been generalized to $\H^3$, \cite{DSZ-B}, but it can not be directly generalized to $\S^3$.

In this paper, we put into the evidence another difference: each central configuration on $\H^3$ is equivalent to one central configuration on $\H^2_{xyw}$, which will be defined later, whereas in $\S^3$ there are central configurations that are not confined to any two-dimensional sphere. In some sense, the number of central configurations in $\H^3$ is smaller than that of $\S^3$.  When we consider the Wintner-Smale conjecture in $\H^3$ raised in \cite{DSZ-B} asking whether the number of classes of central configurations in $\H^3$ is finite or not for generic $N$ point particles, we only need to study the problem on $\H^2_{xyw}$.

The paper is organized as follows: in Section \ref{basic}, we recall the basic setting of the curved $N$-body problem and the corresponding facts about central configurations; in Section \ref{h3}, we prove the result about central configurations in $\H^3$;  in Section \ref{s3}, we construct a two-parameter family  of three-dimensional central configurations in $\S^3$.

 \section{the curved $N$-body problem  and central configurations}\label{basic}
 \subsection{Equations of motion}

  As done in \cite{Dia12a,Dia13}, the equations will be written in $\R^4$ for $\S^3$ and in the Minkowski space $\R^{3,1}$, for $\H^3$.  For convenience, we understand the two linear spaces as $\R^4$ endowed with two inner products: for two vectors, $\q_1=(x_1,y_1,z_1,w_1)^T$ and $\q_2=(x_2,y_2,z_2,w_2)^T$, the inner products are given by
     \[ \q_1\cdot \q_2 = x_1x_2+y_1y_2+z_1z_2 +\sigma w_1w_2, \]
     where $\sigma=1$ for the Euclidean space and $\sigma=-1$  for the Minkowski space. We define the unit sphere $\S^3$ and the unit hyperbolic sphere $\H^3$ as
    \begin{equation*}
     \begin{split}
    \S^3&:=
             \{(x,y,z,w)^T\in\mathbb R^4\ \! |\ \!
             x^2+y^2+z^2+w^2=1\}\,\ \  {\rm and }\\
   \H^3&:=
                       \{(x,y,z,w)^T\in\mathbb R^{4}\ \! |\ \!
                       x^2+y^2+z^2-w^2=-1, \ w>0\},
     \end{split}
     \end{equation*}
  respectively. We can merge these two manifolds into
    \[  \M^3:=
              \{(x,y,z,w)^T\in\mathbb R^{4}\ \! |\ \!
              x^2+y^2+z^2+\sigma w^2=\sigma, \ {\rm with}\ w>0\ {\rm for}\ \sigma=-1\}. \]

    Given the positive masses $m_1,\dots, m_N$, whose positions are described by  the
    configuration $\q=(\q_1,\dots,\q_N)\in(\M^3)^N$,
    $\q_i=(x_i,y_i,z_i,w_i)^T,\ i=\overline{1,N}$, we define the
    singularity set
     \begin{equation}
       \D=\cup_{1\le i<j\le N}\{\q\in (\M^3)^N\ \! ; \ \! \q_i\cdot\q_j=\pm 1\}.
       \notag\end{equation}
 Let $d_{ij}$ be  the geodesic distance between the point masses $m_i$ and $m_j$, which is computed by
 \[  d_{ij}(\q)= \arccos (\q_i \cdot \q_j) \ {\rm for}\ \S^3, \ \ d_{ij}(\q)= {\rm arccosh} (-\q_i \cdot \q_j) \ {\rm for}\ \H^3.   \]
  The force function $U$ ($-U$ being the potential function) in $(\M^3)^N\setminus\D$ is
    $$
    U(\q):=\sum_{1\le i<j\le N}m_im_j\ctn d_{ij}(\q),
    $$
    where $\ctn (x)$ stands for $\cot (x)$ in $\S^3$ and $\coth (x)$ in $\H^3$. We also introduce two more notations, which unify the trigonometric and hyperbolic   functions,
    \[ \sn (x)= \sin(x) \ {\rm or}\ \sinh(x), \
    \ \csn (x)= \cos(x) \ {\rm or}\ \cosh(x). \]
    Define the kinetic energy as $T(\dot \q)= \sum_{1\le i\le N}m_i \dot \q_i \cdot \dot \q_i$, $\dot \q=(\dot \q_1,...,\dot\q_N)$.  Then the curved $N$-body problem is given by the Lagrange system on $T((\M^3)^N\setminus\D)$, with
    \[  \ L(\q,\dot\q):=T(\dot\q)+U(\q).   \]

 Using variational methods, we obtain the equations of motion in $\S^3$ and in $\H^3$, \cite{DSZ-B}.  Merged into one, they are
 \begin{equation*}
    \begin{cases}
    \ddot\q_i=\sum_{j=1,j\ne i}^N\frac{m_im_j [\q_j-\csn d_{ij}\q_i]}{\sn^3 d_{ij}}-\sigma m_i(\dot{\q}_i\cdot \dot{\q}_i)\q_i\cr
    \q_i\cdot \q_i=\sigma, \ \ \ \ i=\overline{1,N}.
    \end{cases}
    \end{equation*}
 The first part of the acceleration access from the gradient of the force function, $\nabla_{\q_i}U(\q)$,  and we will denote it by $\F_i$.  It is the sum of $\F_{ij}=\frac{m_im_j [\q_j-\csn d_{ij}\q_i]}{\sn^3 d_{ij}}$ for $j\ne i$.

 \subsection{Central configurations }
 \begin{definition}
 A configuration $\q\in (\M^3)^N\setminus \D$ is called a \emph{central configuration} if there is some constant $\lambda$ such that
 \begin{equation}\label{CCE}
 \nabla_{\q_i}U(\q)=\lambda\nabla_{\q_i} I(\q),\ i=\overline{1,N},
 \end{equation}
 where $\nabla$ is the gradient operator in $\M^3$, $I(\q)= \sum _{i=1}^N m_i (x_i^2+y_i^2)$, and the explicit form of $\nabla_{\q_i} I$ is

 \begin{equation}\label{I}
    2m_i
         \begin{bmatrix}
          x_i(w_i^2+ z_i^2)\\
          y_i(w_i^2+  z_i^2)\\
         -  z_i(x_i^2+ y_i^2)\\
          -  w_i(x_i^2+ y_i^2)
          \end{bmatrix}{\rm in}\ \ \!  T(\S^3)^N\ \ {\rm and}\ \  2m_i
                  \begin{bmatrix}
                   x_i(w_i^2- z_i^2)\\
                   y_i(w_i^2- z_i^2)\\
                   z_i(x_i^2+ y_i^2)\\
                    w_i(x_i^2+ y_i^2)
                   \end{bmatrix} {\rm in}\ \ \!  T(\H^3)^N.
 \end{equation}

 \end{definition}

 Since the two functions $U$ and $I$ are both invariant under the group action of $SO(2)\times SO(2)$ (in the case of $\S^3$) and $SO(2)\times SO(1,1)$ (in the case of $\H^3$), it is easy to check that a central configuration remains a central configuration after an $SO(2)\times SO(2)$ action (or an $SO(2)\times SO(1,1)$ action), \cite{DSZ-B}. Two central configurations are said to be \emph{equivalent} if one can be transformed to the other by these group actions. When we say a central configuration, we mean a class of central configurations as defined by the equivalence relation.

 A central configuration with $\lambda=0$ is called a \emph{special central configuration}, which only occurs in $\S^3$, \cite{Dia12a}. Otherwise, it is called an \emph{ordinary central configuration}. A central configuration lying on a geodesic  is called a \emph{geodesic central configuration}. A central configuration lying on a two-dimensional sphere  is called an \emph{$\S^2$ central configuration},  a central configuration lying on a two-dimensional hyperbolic sphere  is called an \emph{$\H^2$ central configuration}. All the other central configurations are called \emph{three-dimensional central configurations}.

 Here, a two-dimensional  sphere (hyperbolic sphere) means a sphere (hyperbolic sphere)  isometric to the unit sphere (hyperbolic sphere) in $\R^3$ ($\R^{2,1}$). It is  the non-empty intersection of $\M^3$  with a three-dimensional linear subspace: $\{ (x,y,z,w)^T\in \R^4| ax+by+cz+dw=0 \}$, \cite{BH99}.
  We begin with the following result.

 \begin{proposition}\label{gradient_of_I} Let $V=\{ (x,y,z,w)^T\in \R^4| cz+\sigma dw=0 \}$.
 If a configuration $\q=(\q_1, \cdots, \q_N)$ lies on the two-dimensional  sphere (hyperbolic sphere): $ V \cap \M^3$,   then $\nabla_{\q_i} I$ is in $V$ for  $i=\overline{1,N}$.
 \end{proposition}

 \begin{proof}
 By the explicit form of $\nabla_{\q_i} I$, equation \eqref{I}, we get
 \[ \nabla_{\q_i} I \cdot (0,0,c,d)^T= -\sigma 2m_i (x_i^2+y_i^2)(cz_i+\sigma dw_i)=0. \]
 This equation completes the proof.
 \end{proof}

 \section{Central configurations on $\H^3$}\label{h3}
Let us define $\H_{xyw}^2:=\{(x,y,z,w)^T\in \R^4|  z=0\} \cap \H^3$. We can prove the following result.

 \begin{theorem}\label{theorem_on_h3}
    Each central configuration in $\H^3$ is equivalent to some central configuration on $\H^2_{xyw}$.
    \end{theorem}

\begin{proof}
 We first show that all central configurations in $\H^3$ must lie on a two-dimensional hyperbolic sphere. Then we show that there is some action $\chi\in SO(2)\times SO(1,1) $ which transforms that hyperbolic sphere to $\H^2_{xyw}$. Thus by the definition of equivalent central configurations, each central configuration in $\H^3$ is equivalent to some central configuration on $\H^2_{xyw}$.

 Consider the two-dimensional hyperbolic sphere: $\H^2_{\phi}:=\{ (x,y,z,w)^T\in \R^4| \cosh \phi z- \sinh \phi w=0 \}\cap \H^3 $. The intersection is not empty, since the linear subspace and $\H^3$ share the point $(0,0, \sinh \phi,\cosh \phi )^T$. We show that each central configuration will be confined to only one such two-dimensional hyperbolic sphere.

 Assume that this is not the case. Suppose that there is a  central configuration $\q=(\q_1, \cdots, \q_N)$ with $\q_i \in \H^2_{\phi_i}$,  $\phi_1\ge \phi_i$ for $i\ne 1$ and there is at least one $i$ such that $\phi_1> \phi_i$. Then $\q_i$ can be written as $(x_i,y_i,\rho_i\sinh \phi_i,\rho_i\cosh \phi_i )^T$ with $\rho_i>0$ since $w_i=\rho_i\cosh \phi_i>0$. By Proposition \ref{gradient_of_I}, $\nabla_{\q_1}I$ is in the linear subspace $\{ (x,y,z,w)^T\in \R^4| \cosh \phi_1 z- \sinh \phi_1 w=0 \}$. In order to have a central configuration,  $\nabla_{\q_1}U$ must be in the linear subspace, i.e.,
 \[ \nabla_{\q_1} U \cdot (0,0, \cosh \phi_1,\sinh \phi_1)^T=\F_{1z} \cosh \phi_1 - \F_{1w} \sinh\phi_1  =0,   \]
 where $\F_{1z}$ and $\F_{1w} $ stand for the $z$-coordinate  and $w$-coordinate of $\F_1$, respectively.  However, using the explicit form of $\F_1$, we get
 \begin{equation}
  \begin{split}
  &\ \ \ \ \F_{1z} \cosh \phi_1 - \F_{1w} \sinh\phi_1\\
  &= \sum _{i=2}^{N} m_im_1\left(\frac{z_i-\cosh d_{i1} z_1}{\sinh^3 d_{i1}}\cosh \phi_1 -   \frac{w_i-\cosh d_{i1} w_1}{\sinh^3 d_{i1}}\sinh \phi_1\right)\\
  &=   \sum _{i=2}^{N} m_im_1\frac{\rho_i\sinh \phi_i\cosh \phi_1-\rho_i\cosh  \phi_i\sinh \phi_1-\cosh d_{i1} (z_1\cosh \phi_1 - w_1\sinh \phi_1) }{\sinh^3 d_{i1}}\\
  &=  \sum _{i=2}^{N} m_im_1\frac{ \rho_i\sinh (\phi_i - \phi_1)}{\sinh^3 d_{i1}}<0,
  \end{split}
  \notag
  \end{equation}
since $\phi_i \le \phi_1$ for $i\ne 1$ and there is at least one $i$ such that $\phi_i<\phi_1$.

Thus any central configuration must lie on only one such hyperbolic sphere, say $\H^2_{\phi}$. Let $$\chi=(\begin{bmatrix}
     1 & 0 \\
     0  &1 \end{bmatrix}, \begin{bmatrix}
     \cosh \phi & -\sinh \phi \\
     -\sinh\phi  &\cosh \phi  \end{bmatrix})\in  SO(2)\times SO(1,1).$$
Since $\begin{bmatrix}
     \cosh \phi & -\sinh \phi \\
     -\sinh\phi  &\cosh \phi  \end{bmatrix} \begin{bmatrix}
          \rho_i\sinh \phi \\
          \rho_i\cosh \phi  \end{bmatrix} =\begin{bmatrix} 0 \\
                    \rho_i\end{bmatrix},$
$\chi (\H^2_\phi)= \H^2_{xyw}$. This calculation completes the proof.
 \end{proof}

 To offer more insight into this result, we provide a heuristic argument.
 Recall that the Poincar\'{e} ball model of $\H^3$ is
 \[ \left(  \bar{x}^2+\bar{y}^2+\bar{z}^2<1, \ ds^2=\frac{4(d\bar{x}^2+d\bar{y}^2+d\bar{z}^2)}{(1-\bar{x}^2-\bar{y}^2-\bar{z}^2)^2}  \right).  \]
 In this model, a two-dimensional hyperbolic sphere is the intersection of the  three-dimensional ball with a two-dimensional Euclidean sphere that orthogonally intersects the boundary of the ball.
 The hyperbolic spheres $\H^2_\phi$ defined above are those that intersect the $\bar{z}$-axis orthogonally, \cite{BH99}.
 For example, $\H^2_{xyw}$ in this model is the disk in the plane $\bar{z}=0$.  Now suppose that   $\q_i \in \H^2_{\phi_i}$ and  $\phi_1> \phi_2$. Proposition \ref{gradient_of_I} implies that $\nabla_{\q_1}I \in T_{\q_1} \H^2_{\phi_1}$, as showed in Figure \ref{fig:ball}. However, $\F_{12}$ points towards the lower hyperbolic sphere $\H^2_{\phi_2}$. Thus $\nabla_{\q_1}I$ and $\nabla_{\q_1}U$ cannot be collinear.
 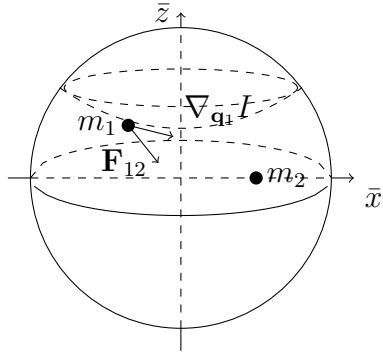
\begin{figure}[!h]
                  \centering
                 \begin{tikzpicture}

      \draw (0,  1.3) circle (2);
         \draw [dashed] (0, -0.7)--(0, 3.3);
     \draw [dashed] (-2, 1.3)--(2, 1.3);
        \draw [->] (0, 3.3)--(0, 3.5) node (zaxis) [left] {$\bar{z}$};
          \draw [->] (2, 1.3)--(2.3, 1.3) node (zaxis) [below right] {$\bar{x}$};
 \draw [-] (-2.3, 1.3)--(-2, 1.3);
 \draw [-] (0,-1)--(0,-0.7);
      \draw [dashed][domain=-2:2] plot (\x, {1.3+sqrt(4-\x*\x)/4});
      \draw [domain=-1.95:1.95] plot (\x, {1.3-sqrt(4-\x*\x)/4});

       \draw[dashed]  (-1.61,2.5) to [out=325,in=215] (1.61,2.5);
       \draw [dashed] (0,2.5) ellipse (1.55  and .25);

        \fill (-0.7,2) circle (2.5pt) node[left] {$m_1$};
      \fill (1,1.3) circle (2.5pt) node[right] {$m_2$};
      \draw [->]  (-0.7,2)--(-0.3, 1.5) node  [left] {$\F_{12}$};
      \draw [->]  (-0.7,2)--(-0.1, 1.85) node  [above right] {$\nabla_{\q_1}I $};

                 \end{tikzpicture}
                 \caption{ A configurations in the  Poincar\'e ball model.} \label{fig:ball}
                   \end{figure}

 \begin{remark}
Recall that there are central configurations not in a plane (called spatial central configurations) in the Newtonian $N$-body problem, such as the regular tetrahedron for any given four masses.
However, those spatial configurations do not lead to rigid motions. Thus if we defined central configurations  in $\R^3$ as those that lead to rigid motions,  there would be no spatial ones.
\end{remark}

  \section{Central configurations in $\S^3$}\label{s3}
Apparently, the compactness of $\S^3$ makes the set of central configuration in it richer than in $\H^3$.
With computations similar to the ones we performed in $\H^3$, we can get the following  necessary conditions for  central configurations in $\S^3$, 
\[\sum _{j=1, j\ne i}^{N} m_im_j\frac{ \rho_j\sin (\phi_j - \phi_i)}{\sin^3 d_{ij}}=0,\ i=\overline{1,N}.\]
These equations, however, do not rule out the existence of three-dimensional central configurations.  For example, we have the  pentatope special central configuration of five equal masses, \cite{DSZ-B},
 \begin{align*}
          x_1&=1,& y_1&=0,& z_1&=0,& w_1&=0, \displaybreak[0]\\
           x_2&=-1/4,& y_2&=\sqrt{15}/4,& z_2&=0,& w_2&=0, \displaybreak[0]\\
           x_3&=-1/4,& y_3&=-\sqrt{5}/(4\sqrt{3}),& z_3&=\sqrt{5}/\sqrt{6},& w_3&=0,\displaybreak[0]\\
          x_4&=-1/4,& y_4&=-\sqrt{5}/(4\sqrt{3}),& z_4&=-\sqrt{5}/(2\sqrt{6}),& w_4&=\sqrt{5}/(2\sqrt{2}),  \displaybreak[0]\\
           x_5&=-1/4,& y_5&=-\sqrt{5}/(4\sqrt{3}),& z_5&=-\sqrt{5}/(2\sqrt{6}),& w_5&=-\sqrt{5}/(2\sqrt{2}).
          \end{align*}
 However, all known three-dimensional central configurations are special central configurations (i.e., $\lambda=0$). We will further construct a two-parameter family of ordinary three-dimensional central configurations of five masses. Suppose that the masses are $m_1=m_2=m, m_3=m_4=m_5=1$, and their positions are given by
 \begin{align*}
          x_1&=0,& y_1&=0,& z_1&=\cos \theta,& w_1&=\sin \theta, \displaybreak[0]\\
          x_2&=0,& y_2&=0,& z_2&=\cos \theta,& w_2&=-\sin \theta, \displaybreak[0]\\
          x_3&=r,& y_3&=0, & z_3&=c,&  w_3&=0,\displaybreak[0]\\
          x_4&=r\cos  \frac{2\pi}{3},& y_4&=r\sin  \frac{2\pi}{3}, & z_4&=c,&  w_4&=0,\displaybreak[0]\\
          x_5&=r\cos  \frac{4\pi}{3},& y_5&=r\sin  \frac{4\pi}{3}, & z_5&=c,&  w_5&=0,
          \end{align*}
where $c\in (-1,1)\setminus\{0\}$, $r>0$, $r^2+c^2=1$ and $\th\in (0, \pi)\setminus \{\frac{\pi}{2}\}$. Such configurations depend on two parameters, $c$ and $\th$, and we denote them by $\q(c, \th)$. It is easy to see that these configurations are not confined to any two-dimensional sphere. In  Figure \ref{fig:s3}, we illustrate such a configuration in a $\R^3$ hyperplane  by the stereographic projection of $\S^3$ from $(0,0,1,0)$ onto the corresponding equatorial $\R^3$ hyperplane, i.e.,  
\[ \bar{x}=\frac{x}{1-z}, \ \ \bar{y}=\frac{y}{1-z}, \ \bar{w}=\frac{w}{1-z}. \]
\begin{figure}[!h]
           	\centering
\begin{tikzpicture}
\draw [->] (0, -0.8)--(0, 3.5) node (zaxis) [left] {$\bar{w}$};
\draw [->] (-2, 1.3)--(2, 1.3) node (xaxis) [below right] {$\bar{x}$};
\draw [->] (-1.7, 0.1)--(1.7, 2.5) node (yaxis) [right] {$\bar{y}$};
\draw [-] (0,1.3) ellipse (1.5 and .45);
\draw (0,  1.3) circle (1.5);

\draw [->]  (-2,2.5)--(-1.5, 1.85) node at (-2,2.5) [above] {$\bar{x}^2 +\bar{y}^2 + \bar{w}^2 =1 $};
\fill (0,3.2) circle (2.5pt) node[right]  {$m_1$};
\fill (0,-0.6) circle (2.5pt) node[right]  {$m_2$};
 \fill (1,1.3) circle (2.5pt) node[right] {$m_3$};
 \fill (-0.3,1.6) circle (2.5pt) node[above left] {$m_4$};
  \fill (-0.8,1.1) circle (2.5pt) node[below left] {$m_5$};

  \end{tikzpicture}
 \caption{ A configuration $\q(c, \th)$ with $(c,\th) \in (-1,0)\times ( 0, \frac{\pi}{2})$.} \label{fig:s3}
                   \end{figure}
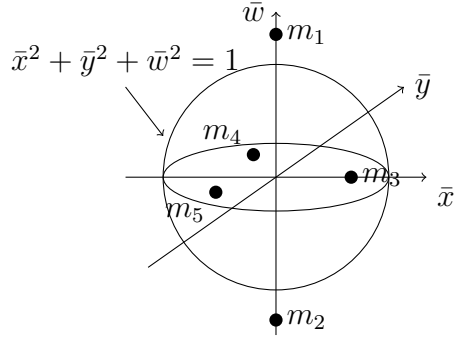
\begin{proposition}
For any $(c,\th) \in (-1,0)\times ( 0, \frac{\pi}{2})$, and $(c,\th) \in (0,1)\times ( \frac{\pi}{2}, \pi)$, the configurations  $\q(c, \th)$  constructed above   are  central configurations if
\begin{equation} \label{condition}
 m= -\frac{3c |\sin^3 2\th|}{2\cos \th (1-c^2 \cos^2 \th)^{3/2}}.
 \end{equation}
Generally,  they are   ordinary central configurations.
\end{proposition}

\begin{proof}
We check that the central configuration equations $\nabla _{\q_i} U= \lambda \nabla _{\q_i} I, i=1, \cdots, 5$,  are satisfied.
The function $U$ can be written as $U=U_1+U_2$, where
\[   U_1= \cot d_{34} + \cot d_{45}+\cot d_{35} ,\
   U_2= m^2 \cot d_{12} +m\sum_{i=3}^5 (\cot d_{1i} +  \cot d_{2i}).\]
 Note that the three equal masses $m_3$, $m_4$, and $m_5$ form an ordinary central configuration themselves, i.e.,  $\nabla _{\q_i} U_1= \lambda_1 \nabla _{\q_i} I$, for $i= 3,4,5$,  $\lambda_1= \frac{-3}{2\sin^3 d_{34}}$, \cite{DSZ-B}.  Note that $\nabla_{\q_1}I =\nabla_{\q_2}I=0$ by equation \eqref{I}. Thus $\nabla _{\q_i} U= \lambda \nabla _{\q_i} I$ is satisfied  if and only if there is some constant $\lambda_2$ such that
 \[ \nabla _{\q_i} U_2= \lambda_2 \nabla _{\q_i} I, i=3,4,5, \ \ {\rm and }\ \ \F_1=\F_2=0.  \]
 By symmetry, we only need to check $\nabla _{\q_3} U_2= \lambda_2 \nabla _{\q_3} I$, and  $\F_1=0$.

Note that $d_{13}=d_{23}=d_{14}=d_{24}=d_{15}=d_{25}$, $d_{34}=d_{45}=d_{35}$,
and
\[ \cos d_{12} = \cos 2\th, \ \cos d_{13} =c\cos \theta, \ \ \cos d_{34} =\frac{3}{2}c^2-\frac{1}{2}. \]
Some straightforward computation shows
\begin{align*}
   \nabla _{\q_3} U_2&= \F_{31}+\F_{32}= \frac{m(\q_1-\cos d_{13}\q_3)}{\sin^3 d_{13}} + \frac{m(\q_2-\cos d_{23}\q_3)}{\sin^3 d_{23}}\\
   &= \frac{m}{\sin^3 d_{13}}(\q_1+\q_2- 2\cos d_{13} \q_3)=\frac{m}{\sin^3 d_{13}} \left( (0,0,2\cos \th,0)^T-2c\cos \th  ( r,0,c,0)^T\right) \\
   &=\frac{-2mr \cos \th}{\sin^3 d_{13}} (c, 0, -r, 0)^T
   \end{align*}
Using equation \eqref{I}, we obtain $\nabla _{\q_3}I=2rc(c, 0, -r, 0)^T$. Thus we can write that
\[ \nabla _{\q_3} U_2= \lambda_2 \nabla _{\q_3} I, \   \lambda_2=\frac{-m \cos\th}{ c \sin^3 d_{13}}. \]
By direct computation, we obtain
\begin{equation}
 \begin{split}
 \F_1&= \F_{12}+ \sum _{j=3}^5\F_{1j}
= \frac{m^2}{|\sin ^3 2\th|}(\q_2-\cos 2\th \q_1) + \sum _{i=3}^5\frac{m}{\sin ^3 d_{13}}(\q_i- \cos d_{13}\q_1)\\
   &= \frac{m^2}{|\sin ^3 2\th|}(\q_2-\cos 2\th \q_1) + \frac{m}{\sin ^3 d_{13}}(\sum _{i=3}^5\q_i- 3c \cos \th \q_1)\\
   &= m \sin \th \left(  \frac{2m \cos \th }{|\sin ^3 2\th|} +
      \frac{3c}{\sin ^3 d_{13}} \right) (0,0, \sin \th, -\cos \th)^T.
 \end{split}
 \notag
 \end{equation}
 Thus $\F_1=0$ if and only if $m= -\frac{3c |\sin^3 2\th|}{2\cos \th (1-c^2 \cos^2 \th)^{3/2}}$. Since we need positive masses, $c \cos \th$ needs to be negative.

 We have thus obtained a two-parameter family of central configurations $\q(c, \th)$ for any $(c,\th) \in (-1,0)\times ( 0, \frac{\pi}{2})$, and $(c,\th) \in (0,1)\times ( \frac{\pi}{2}, \pi)$. The central configuration equations $\nabla _{\q_i} U= \lambda(c, \th) \nabla _{\q_i} I, i=1, \cdots, 5$,  are satisfied,  and the constant is 
 \begin{align*}
 \lambda (c, \th)&= \lambda_1+\lambda_2=\frac{-3}{2\sin^3 d_{34}} - \frac{m \cos \th}{c \sin^3 d_{13}}= \frac{-3}{2\sin^3 d_{34}} + \frac{3|\sin^3 2\th|} {2 \sin^6 d_{13}} \\
               &= \frac{3}{2}\left( \frac{-8}{3\sqrt{3}(1+3c^2)^{3/2} (1-c^2)^{3/2} }  + \frac{|\sin^3 2\th|}{ (1-c^2 \cos^2 \th)^{3}} \right),
 \end{align*}
  which is zero on a one-dimensional manifold. Factually, it is homeomorphic to two open unit intervals.
  Thus generally, $\q(c, \th)$ are ordinary central configurations.  This remark completes the proof.
 \end{proof}

 Moreover, if $(c,\th) \in (-1,0)\times ( 0, \frac{\pi}{2})$, then the masses $m_3, m_4, m_5$ are contained in the unit ball, $\bar{x}^2 +\bar{y}^2 + \bar{w}^2 \le1$, and the masses $m_1, m_2$ are outside, see Figure \ref{fig:s3}.  This happens because
 \[  \bar{w}_1 = \frac{w_1}{1-z_1}= \frac{\sin \th}{1-\cos \th}>1,\
  \bar{x}_3^2 +\bar{y}_3^2= (\frac{x_3}{1-z_3})^2+ (\frac{y_3}{1-z_3})^2=\frac{1+c}{(1-c)}<1. \]
 Similarly, if $(c,\th) \in (0,1)\times (  \frac{\pi}{2}, \pi)$, then masses $m_3, m_4, m_5$ are outside, but the masses $m_1, m_2$ are inside the ball.

 Obviously, we can still obtain central configurations if we substitute the equilateral triangle by  a regular $n$-gon with equal masses,  and generally they are ordinary ones. \\

\textbf{Acknowledgements.} The authors thank F. Diacu for reading the original manuscript and making many useful suggestions. Suo Zhao is supported by NSFC 11501530 and NSFC 11571242. Shuqiang Zhu is funded by a University of Victoria Scholarship and a David and Geoffery Fox Graduate Fellowship.


\begin{thebibliography}{FF}
\bibitem{BH99} M. R. Bridson, A. Haefliger,
               \textit{Metric Spaces of Non-Positive Curvature.}
               Grundlehren der Mathematischen Wissenschaften [Fundamental Principles of Mathematical Sciences], 319. Springer-Verlag, Berlin, 1999.

\bibitem{DPS05} F. Diacu, E. P\'erez-Chavela, M. Santoprete,
                \textit{Saari's Conjecture for the Collinear $N$-Body Problem.}
                Trans. Amer. Math. Soc. 357 (2005), no. 10, 4215-4223. 

\bibitem{Dia12a} F. Diacu,
                 \textit{Relative Equilibria of the Curved $N$-Body Problem},
                 Atlantis Studies in Dynamical Systems, 1. Atlantis Press, Paris, 2012. 

\bibitem{Dia12b} F. Diacu,
                 \textit{Polygonal Homographic Orbits of the Curved $N$-Body Problem.}
                 Trans. Amer. Math. Soc. 364 (2012), no. 5, 2783-2802. 

\bibitem{Dia13} F. Diacu,
                \textit{Relative Equilibria in the 3-Dimensional Curved $N$-Body Problem.}
                Mem. Amer. Math. Soc. 228 (2014), no. 1071. 


\bibitem{Dia-A} F. Diacu,
                  \textit{Bifurcations of the Lagrangian Orbits from the Classical to the Curved 3-body Problem},
                  arXiv:1508.06043.

\bibitem{DSZ-A} F. Diacu, J.M. S\'{a}nchez-Cerritos, and S.Q.\ Zhu,
                  \textit{Stability of Fixed Points and Associated Relative Equilibria of the 3-Body Problem on $S^1$ and $S^2$},
                  arXiv:1603.03339.

\bibitem{DSZ-B} F. Diacu, C. Stoica, and S. Q. Zhu,
               \textit{Central Configurations of the Curved $N$-Body Problem},
               arXiv:1603.03342.

\bibitem{Kil98} A. A. Kilin,
                \textit{Libration Points in Spaces $S^2$ and $L^2$},
                Regul. Chaotic Dyn. 4 (1999), no. 1, 91-103. 
                
 \bibitem{Lap} P.S.  Laplace, \textit{Oeuvres}, vol.\ 4, pp.\ 307--513, vol.\ 11, pp.\ 553--558.

\bibitem{Moe-A} R. Moeckel,
                \textit{Celestial Mechanics---Especially Central Configurations},
                unpublished lecture notes: http://www.math.umn.edu/rmoeckel/notes/CMNotes.pdf

\bibitem{Sma98} S. Smale,
                 \textit{Mathematical Problems for the Next Century},
                Math. Intelligencer 20 (1998), no. 2, 7-15.
\bibitem{Zhu14} S. Q. Zhu,
               \textit{Eulerian Relative Equilibria of the Curved 3-Body Problems in $S^2$.}
                Proc. Amer. Math. Soc. 142 (2014), no. 8, 2837-2848. 
 	
\end{thebibliography}
\end{document}